\documentclass[12pt,english]{article}
\usepackage[T1]{fontenc}
\usepackage[latin9]{inputenc}
\usepackage{geometry}
\usepackage{babel}
\usepackage{mathtools}
\usepackage{url}
\usepackage{amsmath}
\usepackage{amsthm}
\usepackage{amssymb}
\usepackage[unicode=true,pdfusetitle,
 bookmarks=true,bookmarksnumbered=false,bookmarksopen=false,
 breaklinks=false,pdfborder={0 0 1},backref=false,colorlinks=false]
 {hyperref}

\makeatletter
\numberwithin{equation}{section}
\numberwithin{figure}{section}
\theoremstyle{plain}
\newtheorem{thm}{\protect\theoremname}
\theoremstyle{plain}

\theoremstyle{plain}
\newtheorem{lem}[thm]{\protect\lemmaname}
\theoremstyle{remark}
\newtheorem{rem}[thm]{\protect\remarkname}
\theoremstyle{plain}
\newtheorem{prop}[thm]{\protect\Propositionname}
\theoremstyle{remark}

\theoremstyle{definition}

\theoremstyle{plain}
\newtheorem{cor}[thm]{Corollary}

\date{}
\usepackage{tikz-cd}
\usepackage{wasysym}
\usepackage{MnSymbol}
\usepackage[tt=false]{libertine}
\usepackage[parfill]{parskip}
\usepackage{enumitem}
\setlist[itemize]{noitemsep,topsep=5pt}

\usepackage{titlesec}
\titleformat{\section}{\large\bfseries\filleft}{\thesection}{1em}{}[{\titlerule[0.8pt]}]

\renewcommand\labelenumi{(\roman{enumi})}
\renewcommand\theenumi\labelenumi

\DeclareMathOperator{\coker}{coker}

\DeclareMathOperator{\codim}{codim}

\DeclareMathOperator{\Aut}{Aut}

\DeclareMathOperator{\Ext}{Ext}

\DeclareMathOperator{\im}{im}

\DeclareMathOperator{\Mor}{Mor}

\DeclareMathOperator{\PGL}{PGL}

\let\oldtheorem\thm
\renewcommand{\thm}{\oldtheorem\normalfont}
\let\oldprop\prop
\renewcommand{\prop}{\oldprop\normalfont}
\let\oldcor\cor
\renewcommand{\cor}{\oldcor\normalfont}
\let\oldlem\lem
\renewcommand{\lem}{\oldlem\normalfont}
\newenvironment{ack}{\textit{Acknowledgements.}}{}
\makeatother

\providecommand{\claimname}{Claim}
\providecommand{\definitionname}{Definition}
\providecommand{\lemmaname}{Lemma}
\providecommand{\Propositionname}{Proposition}
\providecommand{\questionname}{Question}
\providecommand{\remarkname}{Remark}
\providecommand{\theoremname}{Theorem}

\begin{document}
\global\long\def\A{\mathbb{A}}%

\global\long\def\C{\mathbb{C}}%

\global\long\def\E{\mathbb{E}}%

\global\long\def\F{\mathbb{F}}%

\global\long\def\G{\mathbb{G}}%

\global\long\def\H{\mathbb{H}}%

\global\long\def\N{\mathbb{N}}%

\global\long\def\P{\mathbb{P}}%

\global\long\def\Q{\mathbb{Q}}%

\global\long\def\R{\mathbb{R}}%

\global\long\def\O{\mathcal{O}}%

\global\long\def\Z{\mathbb{Z}}%

\global\long\def\ep{\varepsilon}%

\global\long\def\laurent#1{(\!(#1)\!)}%

\global\long\def\wangle#1{\left\langle #1\right\rangle }%

\global\long\def\ol#1{\overline{#1}}%

\global\long\def\mf#1{\mathfrak{#1}}%

\global\long\def\mc#1{\mathcal{#1}}%

\global\long\def\norm#1{\left\Vert #1\right\Vert }%

\global\long\def\et{\textup{ét}}%

\global\long\def\Et{\textup{Ét}}%
\title{Non-smoothness of Moduli Spaces of Higher Genus Curves on Low Degree Hypersurfaces}
\author{Matthew Hase-Liu and Amal Mattoo}

\maketitle
\begin{abstract}
We show that the moduli space of degree $e$ maps from smooth genus $g\ge 1$ curves to an arbitrary low degree smooth hypersurface is singular when $e$ is large compared to $g$. We also give a lower bound for the dimension of the singular locus. 
\end{abstract}
\tableofcontents{}

\section{Introduction}
Let $X$ be a smooth complex hypersurface in $\P^n$ of degree $d$. For $d=1$ and $d=2$ (and more generally for a convex variety), Kontsevich famously showed in \cite{kont} that the compactified moduli space $\overline{\mathcal{M}}_{0,k}\left(X, e\right)$ of rational curves on $X$ of degree $e$ and $k$ marked points is smooth. 

A natural question is whether similar behavior is exhibited for moduli spaces of curves on hypersurfaces, where both the genus of the curves and the degrees of the hypersurfaces are at least one. 

We work exclusively over $\mathbb{\C}$. Denote by $\mathcal{M}_g\left(X, e\right)$ the moduli space of degree $e$ maps from smooth projective genus $g$ curves to $X$. Since the $g=0$ case is relatively well-understood, we focus on the $g\ge 1$ case. Note that the compactified Konstevich space $\overline{\mathcal{M}}_g(X,e)$ is almost always singular.

Our main result is as follows:

\begin{thm}\label{mainthm}
    Let $e,g,d,n$ be positive integers and $X$ a smooth hypersurface in $\P^n$ of degree $d$. Then $\mathcal{M}_{g}(X,e)$ is not smooth in the following cases:
    
    \begin{tabular}{l l}
         $d=2$: & If $g=1,$ $e\ge 14$, $n\ge 5$ or $g\ge 2$, $e\ge 35g/2 - 15/2$, $n\ge 5$.\\
         $d\geq 3$: & If $e\ge \max\left(2g-1, g+2\right)$ and $n\ge d-1$.
    \end{tabular}
    
    Moreover, the dimension of the singular locus is at least $2n+2e+2g-7$ for $d=2$ and at least $n+2e+2g-5$ for $d\ge 3$.
\end{thm}
\begin{rem}
    Under the assumptions of Theorem \ref{mainthm}, for an arbitrary smooth hypersurface, the moduli spaces $\mathcal{M}_{g}(X,e)$ are only known to be irreducible and have the expected dimension for $n$ exponentially large compared to $d$ by Theorem 1 of \cite{haseliu2024higher}. In particular, for $d\ge 3$, we are able to establish non-smoothness and a lower bound on the singular locus without knowing the reducibility or dimension of these moduli spaces.
\end{rem}

    

In Section \ref{smooth}, we also give examples of cases where $\mathcal{M}_{g}(X,e)$ is actually smooth. 

\begin{prop}\label{smoothcase}
    Let $e,g,d,n$ be positive integers and $X$ a smooth hypersurface in $\P^{n}$ of degree $d$. Then $\mathcal{M}_{g}(X,e)$ is smooth in the following cases: 

    \begin{tabular}{l l}
         $d=1$: & If $e\geq 0$ and $n=2$ (Proposition \ref{smooth_P1}) or $e \ge 2g-1$ and $n\geq 3$ (Proposition \ref{smooth_d=1}).\\
         $d=2$: & If $e\geq 0$ and $n=2$ (Proposition \ref{smooth_P1}) or $e\ge 4g-3$ and $n=3$ (Corollary \ref{smooth_n=3}).
    \end{tabular}
\end{prop}

Our original motivation for writing this note was understanding the $d=2$ case for higher genus curves. Work in progress by the first author and Jakob Glas aims to investigate terminality of $\mathcal{M}_g\left(X,e\right)$ using methods from analytic number theory, which for genus zero is clear for $d=2$ (since the moduli spaces are smooth). 

In the genus zero case, when $n\geq 2d \geq 6$, Harris, Roth, and Starr \cite{HarrisRothStarrRatCurvesI} showed that $\mathcal{M}_{0}(X,e)$ is generically smooth for general $X$. In an unpublished note, Starr and Tian moreover proved a lower bound for the codimension of the singular locus of $\mathcal{M}_{0}(X,e)$ for general Fano hypersurfaces using an inductive bend-and-break strategy. Browning and Sawin \cite{BrowningSawinFree} used the circle method to give better bounds on the codimension of the singular locus of $\mathcal{M}_{0}(X,e)$ for any smooth hypersurface $X$ at the cost of a logarithmic upper bound on $d$ in terms of $n$. 

For $g\ge 1$, aside from the circle method-based approaches of \cite{haseliu2024higher} and forthcoming work of Sawin on the asymptotic for Waring's problem over function fields, in addition to recent work on topics related to Geometric Manin's Conjecture in \cite{GM4, GM3, GM2}, there appear to be relatively few results in the literature about singularities of $\mathcal{M}_g\left(X,e\right)$. Moreover, such results (e.g. \cite{GM3} in particular) focus on establishing a lower bound on the codimension of the singular locus, whereas we prove a lower bound on the dimension instead.

To prove Theorem \ref{mainthm}, we work with the simpler space $\Mor_e\left(C, X\right)$, which is the moduli space of maps of degree $e$ from a fixed smooth projective curve $C$ of genus $g$ to $X$.

More precisely, we show the following:

\begin{prop}\label{reductthm}
    Under the same conditions of Theorem \ref{mainthm} for $d\geq 2$, if $C$ is a general projective curve of genus $g\geq1$, then $\Mor_e\left(C, X\right)$ is not smooth. Moreover, the dimension of the singular locus is at least $2n+2e-g-4$ for $d=2$ and at least $n+2e-g-2$ for $d\ge 3$.
\end{prop}

For $d\ge 2$, let us explain how Theorem \ref{mainthm} reduces to Proposition \ref{reductthm}.

\begin{lem}\label{reductlem}
    Under the same conditions of Theorem \ref{mainthm} for $d\geq 2$, if $\Mor_{e}\left(C,X\right)$ is not smooth for general $C\in \mathcal{M}_g$, then $\mathcal{M}_{g}\left(X,e\right)$ is not smooth. Moreover, if the singular locus of $\Mor_{e}\left(C,X\right)$ has dimension at least $m$, then the singular locus of $\mathcal{M}_g\left(X,e\right)$ has dimension at least $m+3g-3.$ 
\end{lem}
\begin{proof}
Recall that we have a canonical map $\vartheta\colon \mathcal{M}_g\left(X, e\right) \to \mathcal{M}_g$ that on points forgets the map to $X$. In particular, the fiber above a smooth projective curve $C$ is given by $\Mor_e\left(C,X\right)$. Moreover, by Lemma \ref{constructmap} and the non-emptiness of the Fano scheme of lines on $X$ by Theorem 8 of \cite{fanolines}, it follows that $\vartheta$ is dominant.

For the sake of contradiction, suppose $\mathcal{M}_g\left(X, e\right)$ is smooth. By Tag 0E84 of \cite{key-2}, $\mathcal{M}_g$ is smooth. By Exercise 10.40 of \cite{AG}, often called "generic smoothness on the target," it follows that there is an open dense sub-stack $U$ of $\mathcal{M}_g$ such that the restriction $\vartheta^{-1}(U) \to U$ is smooth. Note that we are using the assumption that we are working over $\C$. Moreover, since $\vartheta$ is dominant, the pre-image $\vartheta^{-1}(U)$ is non-empty. In particular, for any closed point $x\in U$ we have $\vartheta^{-1}(x)$ is smooth.

By hypothesis, for a general projective curve $C$ of genus $g$, we know $\Mor_e\left(C,X\right)$ is not smooth. In particular, since the intersection of two dense open sub-stacks is also a dense open sub-stack, it follows that there is a closed point in $U$ such that the fiber of $\vartheta$ above this point is not smooth, which is a contradiction.  

Finally, we lower bound the dimension of the singular locus. Let $V\subset \mathcal{M}_g\left(X,e\right)$ be the smooth locus. Applying generic smoothness on the target to the restriction $\vartheta|_{V}\colon V\to\mathcal{M}_{g}$, a general fiber of $\vartheta|_{V}$ is smooth. Then the singular locus of a general fiber of $\vartheta$ is contained in $\mathcal{M}_{g}(X,e)\setminus V$, so the singular locus of $\mathcal{M}_{g}(X,e)$ intersected with a general fiber has dimension at least $m$. Thus, the result follows from $\dim\mathcal{M}_g = 3g-3$.
\end{proof}

In Section \ref{nonsmoothd=2}, we explain the proof of Theorem \ref{mainthm} for the case $d=2$. Theorem 1 of \cite{haseliu2024higher} shows that the relevant moduli spaces are irreducible and of the expected dimension, so it suffices to construct maps $f\colon C \to X$ such that the restricted tangent bundle $f^* \mathcal{T}_X$ has non-vanishing first cohomology. These maps arise from choosing a line on $X$, then composing with maps of chosen degree $C\to \P^1$. We note that Theorem \ref{mainthm} and Proposition \ref{smoothcase} miss the case $(d,n) = (2,4)$, which we expect a more efficient version of the circle method (as used in \cite{haseliu2024higher}) to show that the relevant moduli space has the expected dimension and hence fail to be smooth by using the same argument for $n\ge 5$. 

In Section \ref{dbigger3}, we explain the proof of Theorem \ref{mainthm} for the case $d\ge 3$. We use a similar technique to construct two maps $f:C\to X$ with different values of $h^{0}(C,f^{*}\mathcal{T}_{X})$ and show that they lie in the same irreducible component of $\Mor_{e}(C,X)$. We make use of the stratification of the Fano scheme of lines on $X$ found in \cite{hannah}.

\ack{ Thank you very much to Johan de Jong and Will Sawin for suggestions to improve our results and to Jason Starr for helpful comments about the literature. Both authors were partially supported by National Science Foundation Grant Number DGE-2036197.
}

\section{Smooth Examples}\label{smooth}
Let $C$ be a smooth projective curve of genus $g\ge 1$ and $X$ a smooth hypersurface in $\P^n$ of degree $d\ge 1$. 

Let us first address the case $d=1$, i.e. when $X=\P^{n-1}$.
\begin{lem}\label{Euler}
   Suppose $e \ge 2g-1.$ For any degree $e$ map $f\colon C \to \P^{n-1}$, we have  $H^1\left(C, f^*\mathcal{T}_{\P^{n-1}}\right)=0$.
\end{lem}

\begin{proof}
Pulling back the Euler exact sequence $0\to \mathcal{O} \to \mathcal{O}(1)^{\oplus{n}} \to \mathcal{T}_{\P^{n-1}} \to 0$ along $f$ and taking cohomology tells us that $H^1\left(C,f^*\mathcal{O}(1)\right)^{\oplus{n}}$ surjects onto $H^1\left(C, f^*\mathcal{T}_{\P^{n-1}}\right)$. But Serre duality gives $h^1\left(C,f^*\mathcal{O}(1)\right) = h^0\left(C, f^*\mathcal{O}(-1) \otimes K_C\right),$ where $K_C$ is the canonical bundle of $C$. The degree of $f^*\mathcal{O}(-1) \otimes K_C$ is $-e+2g-2<0$, which gives $H^1\left(C, f^*\mathcal{T}_{\P^{n-1}}\right)=0$. 

\end{proof}
\begin{lem}\label{fiberd=1smooth}
    If $e \ge 2g-1$ and $g\geq 1$, then $\Mor_e\left(C, \P^{n-1}\right)$ is smooth.
\end{lem}
\begin{proof}
    This follows immediately from Lemma \ref{Euler} and Chapter 2 of \cite{debarre}.
\end{proof}
\begin{prop}\label{smooth_d=1}
    If $e \ge 2g-1$, $g\geq 1$, and $n\ge 3$, then $\mathcal{M}_g\left(\P^{n-1},e\right)$ is smooth. 
\end{prop}
\begin{proof}
    The tangent space to $\mathcal{M}_{g}(\P^{n-1},e)$ at a point $(C,f)$ is 
    $$\mathcal{T}_{(C,f)}\mathcal{M}_{g}(\P^{n-1},e)=R^{1}\Gamma\left(F^\bullet\right),$$ where $F^\bullet$ is the complex $\mathcal{T}_{C}\to f^*\mathcal{T}_{\P^{n-1}}$ placed in degrees 0 and 1, and $R^i\Gamma$ denotes the $i$th hypercohomology vector space. 

    Recall that we have a hypercohomology spectral sequence $E_1^{i,j} = R^j\Gamma\left(F^i\right)\Rightarrow R^{i+j}\Gamma\left(F^\bullet\right)$. Applying Lemma \ref{Euler}, the $E_1$-page in our situation has only three possibly non-zero terms. In particular, we have
    $$\dim\mathcal{T}_{(C,f)}\mathcal{M}_{g}(\P^{n-1},e)=h^{1}(C,\mathcal{T}_{C})+\dim \coker(H^{0}(C,\mathcal{T}_{C})\to H^{0}(C,f^{*}\mathcal{T}_{\P^{n-1}})).$$
    We want to show $H^{0}(C,\mathcal{T}_{C})\to H^{0}(C,f^{*}\mathcal{T}_{\P^{n-1}})$ is injective. If $g>1$ then $H^{0}(C,\mathcal{T}_{C})=0$, and if $g=1$ then $h^{0}(C,\mathcal{T}_{C})=h^{0}(C,\mathcal{O})=1$, so it suffices to show that the map is nonzero in the case $g=1$. But since $f$ is not constant, $\mathcal{T}_{C}\to f^{*}\mathcal{T}_{\P^{n-1}}$ is not the zero map on each fiber, and so the induced map on global sections is nonzero.   

    Thus,
    \begin{align*}
        \dim\mathcal{T}_{(C,f)}\mathcal{M}_{g}(\P^{n-1},e) &=
        h^{1}(C,\mathcal{T}_{C})+h^{0}(C,f^{*}\mathcal{T}_{\P^{n-1}})-h^{0}(C,\mathcal{T}_C)\\&= \chi(C,f^{*}\mathcal{T}_{\P^{n-1}}) -\chi(C,\mathcal{T}_{C}),   
    \end{align*}
    and by Riemann-Roch, the right hand side depends only on the genus of $C$ and the degree of $f$, proving smoothness of $\mathcal{M}_g\left(\P^{n-1},e\right)$. 
\end{proof}

\begin{prop}\label{smooth_P1}
    If $e\geq0$ and $g\geq1$, then $\mathcal{M}_{g}(\P^{1},e)$ is smooth.
\end{prop}
\begin{proof}
    The argument is essentially the same as the proof of Proposition \ref{smooth_d=1}, except that we cannot use Lemma \ref{Euler} to conclude that $H^1\left(C, f^*\mathcal{T}_{\P^1}\right) = 0$ when $e<2g-1$. Instead, our $E_1$-page of the hypercohomology spectral sequence has four possibly non-zero terms. 
    
    Running the same argument as above, it suffices to show that the cokernel of $H^1\left(C, \mathcal{T}_C\right)\to H^1\left(C, f^*\mathcal{T}_{\P^1}\right)$ is zero. By Serre duality, this is equivalent to showing that kernel of the dual map $H^0\left(C, f^*K_{\P^1} \otimes K_C\right) \to H^0\left(C,K_C^{\otimes 2}\right)$ is zero. But we always have an injection of sheaves $f^*K_{\P^1} \to K_C$, which implies we have an injection of sheaves after tensoring by $K_C$, and hence on global sections as well.
\end{proof}
\begin{rem}
    Note that Proposition \ref{smooth_P1} covers the cases of $d=1$ and $d=2$ for $n=2$, since a conic in $\P^{2}$ is isomorphic to $\P^1$. 
\end{rem}
 

For $d=2$ and $n=3$, note that the hypersurface $X$ is isomorphic to $\P^1\times \P^1$. We may reduce to the previous proposition again.
\begin{cor}\label{smooth_n=3}
    If $e \ge 4g-3$, $d = 2$, and $n=3$, then $\mathcal{M}_g\left(X,e\right)$ is smooth.
\end{cor}
\begin{rem}
    In this case $\mathcal{M}_g\left(X,e\right)$ is neither irreducible nor of the expected dimension!
\end{rem}
\begin{proof}[Proof of Corollary \ref{smooth_n=3}]
    Since $X\cong \P^1\times \P^1$, note that $\mathcal{M}_g\left(X,e\right)$ has components parameterized by bidegrees $(i,e-i)$ of maps $C\to \P^1\times \P^1$. For each bidegree $(i,e-i)$, we can view the corresponding component as a projection over $\mathcal{M}_g\left(\P^1,i\right)$ with fibers $\Mor_{e-i}\left(C,\P^1\right)$, or alternatively as a projection over $\mathcal{M}_g\left(\P^1,e-i\right)$ with fibers $\Mor_{i}\left(C,\P^1\right)$. Without loss of generality, suppose $i\ge e-i$, which implies $i\ge 2g-1$. By Proposition \ref{smooth_P1}, $\mathcal{M}_g\left(\P^1,e-i\right)$ is smooth. By Lemma \ref{fiberd=1smooth}, $\Mor_i\left(C, \P^1\right)$ is also smooth. It follows that the total space $\mathcal{M}_g\left(X,e\right)$ is smooth. 
\end{proof}

\section{Non-smoothness of $\mathcal{M}_g\left(X,e\right)$ for $d=2$}\label{nonsmoothd=2}
In this section, let $X\hookrightarrow \P^n$ be a smooth hypersurface of degree $d$ and $C$ a general projective curve of genus $g\ge 1.$ 

We proceed with the proofs of Proposition \ref{reductthm} and Theorem \ref{mainthm} for $d=2$.

\begin{lem}\label{constructmap}
    For any integer $a\ge g/2 +1$, there exists a morphism $C \to \P^1$ of degree $a$. The locus of such maps has dimension $-g+2a-2.$ 
\end{lem}
\begin{proof}
    Theorem V.1.5 of \cite{ACGH} implies that given $a$ and $r$ such that the Brill-Noether number $g-(r+1)(g-a+r)\geq0$, there is a linear series $(D,V)$ on $C$ with $\deg D=a$ and $\dim V=r+1$, i.e. a $g^r_d.$ The same theorem tells us that the locus of such $g^r_d$'s has pure dimension equal to the Brill-Noether number. 
    
    For $r=1$, since $a\geq g/2+1 $ we have $g-2(g-a+1)\geq 0$. Note that $g-2(g-a+1) > g-2(g-(a-1)+1)$, from which it follows that there is a $g^1_d$ that is moreover base-point free. This linear system defines a morphism of degree $a$ and the locus of such maps has dimension $-g+2a-2.$
\end{proof}

\begin{lem}\label{mapexists}
    Assuming the conditions of Proposition \ref{reductthm}, there exists a map $f\colon C\to X$ of degree $e\ge 1$ such that $H^1\left(C,f^*\mathcal{T}_X\right)\ne 0.$
\end{lem}

\begin{proof}
By Theorem 8 of \cite{fanolines}, $X$ contains lines, and suppose $\varphi \colon \P^1 \hookrightarrow X$ is the embedding of one such line in $X$. By the Birkhoff-Grothendieck Theorem, we know $\varphi^* \mathcal{T}_X$ splits as a direct sum of line bundles, say 
\[\varphi^* \mathcal{T}_X\cong \bigoplus_{i=1}^{n-1}\mathcal{O}\left(a_i\right) \] with $a_{n-1}\ge a_{n-2} \ge \cdots \ge a_{1}$.

The adjunction formula tells us that $\det\left(\mathcal{T}_X\right)=\mathcal{O}_X\left(n+1-d\right)$. Then, we have $\deg\left(\varphi^* \mathcal{T}_X\right)=\deg(\varphi) \cdot \deg \left(\mathcal{T}_X\right) = n+1-d.$ Since $d=2$, we obtain $\deg\left(\varphi^*\mathcal{T}_X\right) = n-1.$

Since $\varphi^*\mathcal{T}_X$ has rank $n-1$, we see that the slope of $\varphi^*\mathcal{T}_X = 1$, which implies that $a_1=1$ or $a_1\le 0$. The former is possible only when $a_{n-1}=\cdots = a_1=1$.

Let us deal with this case first, i.e. when $\varphi^* \mathcal{T}_X \cong \mathcal{O}(1)^{\oplus(n-1)}.$  Since $\varphi$ is a line, the differential map $d\varphi \colon \mathcal{T}_{\P^1} \to \varphi^* \mathcal{T}_X$ is non-zero, which means we have a non-zero map $\mathcal{O}(2) \to \mathcal{O}(1)^{\oplus(n-1)},$ which is evidently impossible.

This means that we must be in the situation where $a_1 \le 0$. 

By Lemma \ref{constructmap}, there exists a map, say $\phi\colon C \to \P^1$, such that $\deg\left(\phi\right) = e$. Let us then define $f = \phi \circ \varphi$.

Then, $f^*\mathcal{T}_X$ splits into a direct sum of line bundles as well, so $H^1\left(C, \phi^*\mathcal{O}\left(a_1\right)\right)$ is a direct summand of $H^1\left(C, f^*\mathcal{T}_X\right).$ In particular, it suffices to show that $H^1\left(C, \phi^*\mathcal{O}\left(a_1\right)\right)$ is non-zero.

Suppose momentarily that the genus $g\ge 2$ or that $a_1 \le -1$. Then, by Riemann-Roch, note that the Euler characteristic $\chi\left(C, \phi^*\mathcal{O}\left(a_1\right)\right) = ea_1 - g + 1 < 0$. This implies that $H^1\left(C, \phi^*\mathcal{O}\left(a_1\right)\right)$ is non-zero, and we are done.

Finally, for $g=1$ and $a_1 = 0$, we have $H^1\left(C, \phi^*\mathcal{O}\left(a_1\right)\right) = H^1\left(C, \mathcal{O}\right),$ which is also non-zero by Serre duality and the fact that $H^0\left(C, \mathcal{O}\right) \ne 0$. 
\end{proof}
\begin{rem}
    The same argument above works for $d\ge 3$, but the argument in Section \ref{dbigger3} gives a better bound for $n$ in Theorem \ref{mainthm}.
\end{rem}

\begin{rem}
    We can show that Lemma \ref{constructmap}, and hence Proposition \ref{reductthm}, holds true for an arbitrary smooth projective curve $C$ of genus $g \ge 1$ provided $e \geq \max\left(2g-1, 2\right)$. 
\end{rem}
\begin{proof}[Proof of Proposition \ref{reductthm} for $d=2$]
Theorem 1 of \cite{haseliu2024higher} shows that $\Mor_e\left(C,X\right)$ is irreducible and of the expected dimension $h^{0}(C,f^{*}\mathcal{T}_{X})-h^{1}(C,f^{*}\mathcal{T}_{X})$. However, Lemma \ref{mapexists} shows that there is a map $f\colon C\to X$ of degree $e$ such that $H^1\left(C,f^*\mathcal{T}_X\right) \ne 0,$ which means that the dimension $h^{0}(C,f^{*}\mathcal{T}_{X})$ of the tangent space at $[f]$ is strictly larger than the dimension of $\Mor_e\left(C,X\right)$ at $[f].$ Hence, $\Mor_e\left(C,X\right)$ is not smooth. 

Note that we can bound the dimension of the singular locus from below by the sum of the dimension of the locus of choices of degree $e$ maps $C\to \P^1$ and the dimension of the degree one maps $\P^1\hookrightarrow X$. The former is precisely the Brill-Noether number $-g+2e-2$ by Lemma \ref{mapexists}, and the latter is $2n-5 + 3$ by Theorem 1.3 of \cite{beheshti_riedl} and the fact that $\Aut(\P^{1})=\PGL_{2}$, which has dimension $3$. The result follows.
\end{proof}
\begin{proof}[Proof of Theorem \ref{mainthm} for $d=2$]
    Follows immediately from Proposition \ref{reductthm} for $d=2$ and Lemma \ref{reductlem}.
\end{proof}


    
    
    
    

\section{Non-smoothness of $\mathcal{M}_g\left(X,e\right)$ for $d\ge 3$}\label{dbigger3}
In this section, let $X\hookrightarrow \P^{n}$ be a smooth hypersurface of degree $d$ and $C$ a smooth projective curve of genus $g\geq 1$.

We proceed with the proofs of Proposition \ref{reductthm} and Theorem \ref{mainthm} for $d\ge 3$.

\begin{lem}\label{decomp}
    Let $\varphi\in\Mor_{1}(\P^{1},X)$. The restricted tangent bundle $\varphi^* \mathcal{T}_X$ is isomorphic to $\mathcal{O}(2)\oplus \bigoplus_{i=1}^{n-2}\mathcal{O}\left(a_i\right)$ with $1\ge a_{n-2}\ge a_{n-3} \ge \cdots \ge a_{1}$.
\end{lem}
\begin{proof}
    Recall that the normal bundle of a line in $\P^n$ is isomorphic to $\mathcal{O}(1)^{\oplus(n-1)}$. Write $\mathcal{N}_{\P^1/X}$ as $\bigoplus_{i=1}^{n-2}\mathcal{O}\left(a_i\right)$. Since $\mathcal{N}_{\P^1/X}$ injects into $\mathcal{N}_{\P^1/\P^n}$, it follows that each $a_i$ is at most 1. Consider the normal bundle exact sequence $0\to \mathcal{T}_{\P^1} \to \varphi^*\mathcal{T}_X \to \mathcal{N}_{\P^1/X}\to 0$. Since $\mathcal{T}_{\P^1}\cong \mathcal{O}(2)$ and each $a_i\le 1$, it follows that $\Ext^1\left(\mathcal{N}_{\P^1/X}, \mathcal{T}_{\P^1}\right)\cong H^1\left(\P^1, \bigoplus_{i=1}^{n-2}\mathcal{O}\left(2 - a_i\right)\right)=0$, which means that the exact sequence splits. The result follows.
\end{proof}

\begin{lem}\label{balanced}
    Assuming $d\leq n-1$. On each irreducible component of the space of maps $\Mor_1(\P^{1},X)$, there is a map $\varphi\colon\P^{1}\hookrightarrow X$ such that $a_{1}\geq 0$ (in the notation of Lemma \ref{decomp}).
\end{lem}
\begin{proof}
    This is a small modification of the proof of Proposition 4.9 in \cite{debarre}. 

    Consider the universal space $I=\P^1 \times \Mor_1(\P^1, X)$ and the natural evaluation map $p \colon I\to X$.
    
    By Proposition 2.13 of \cite{debarre}, we know that $X$ is covered by lines, i.e. through any point of $X$, there is a line on $X$ that passes through the point. This means that $p$ is surjective.
    
    Since we are working over $\C$, the evaluation map $p$ is generically smooth on the source, which means that the induced map on tangent spaces $T_{(x,\varphi)}I \to T_{x}X$ is surjective for general $(x,\varphi)$. In particular, each irreducible component of $\Mor_1(\P^1,X)$ contains such a $\varphi$. Let $(x, \varphi)$ be such a pair, and let $t\in\P^{1}$ with $\varphi(t)=x$.
    

    Since the tangent space to $\Mor_1(\P^1,X)$ for a map $\varphi\colon \P^1 \hookrightarrow X$ is given by $H^0\left(\P^1, \varphi^*\mathcal{T}_{X}\right)$, we have $T_{(x,\varphi)}I\cong H^0\left(\P^1, \mathcal{T}_{\P^1}\right)\oplus H^0\left(\P^1, \varphi^*\mathcal{T}_{X}\right)$. So we have shown that $H^0\left(\P^1, \varphi^*\mathcal{T}_X\right)$ surjects onto $T_{x}X/\im\left(d\varphi_t\right)$. Moreover, since $\mathcal{T}_{\P^1}$ is globally generated, the natural map $H^0(\P^1, \mathcal{T}_{\P^1})\to T_{t}\P^1$ is surjective. By the commutative diagram
    $$\begin{tikzcd}
        {H^{0}(\mathbb{P}^{1},\mathcal{T}_{\mathbb{P}^{1}})} \arrow[r, two heads] \arrow[d] & {T_{t}\mathbb{P}^{1}} \arrow[d, "d\varphi_{t}"] \\
        {H^{0}(\mathbb{P}^{1},\varphi^{*}\mathcal{T}_{X})} \arrow[r]                        & {T_{x}X}                                 
    \end{tikzcd}$$
    the image of the map $H^0\left(\P^1, \varphi^*\mathcal{T}_X\right) \to T_{x}X$ contains the image of $d\varphi_t$ by naturality of $d\varphi$. In other words, $H^0\left(\P^1, \varphi^*\mathcal{T}_X\right)\to T_{x}X\cong(\varphi^{*}\mathcal{T}_{X})_{t}$ is surjective, which means $\varphi^*\mathcal{T}_X$ is globally generated at $t$, and therefore globally generated everywhere. Hence, $a_1\ge 0.$

\end{proof}

Let us set some notation, following \cite{hannah}. Letting $F(X)$ be the space of lines $\varphi\colon L\hookrightarrow X$, for $\Vec{a}=(a_{1},...,a_{n-2})$, define a sublocus $F_{\Vec{a}}(X)\subset F(X)$ of lines $\varphi\colon L \hookrightarrow X$ such that $\mathcal{N}_{L/X}\cong \bigoplus_{i=1}^{n-2}\mathcal{O}\left(a_i\right)$. 

Theorem 8 of \cite{fanolines} shows, for $2n\ge d+3$, that the Fano scheme of lines $F(X)$ is non-empty for any hypersurface $X$, and Theorem 1.1 of \cite{hannah} gives $$\codim_{F(X)}F_{\Vec{a}}(X)\le \sum_{i<j}\max\{a_{j}-a_{i}-1,0\}.$$ 


\begin{lem}\label{unbalanced}
    Assuming $n \ge d \ge 3$, there exists a $\varphi\in\Mor_{1}(\P^{1},X)$ such that $a_{1}\leq-1$, and the dimension of the locus of such $\varphi$ is at least $n$.
\end{lem}
\begin{proof}

    Let $\P^{N}$ for $N=\binom{n+d}{d}-1$ be the moduli space of all degree $d$ hypersurfaces and $U\subset \P^N$ be the open subset of smooth degree $d$ hypersurfaces. Also, let $$\Sigma=\left\{(L,X)\in F(\P^n) \times U\colon L\subset X\right\}$$ be the universal Fano scheme of lines over smooth degree $d$ hypersurfaces. For a vector $\Vec{a}=\left(a_1,\ldots, a_{n-2}\right)$, let $\Sigma_{\Vec{a}}$ denote the sublocus of $\Sigma$ comprising $(L,X)$ such that $\mathcal{N}_{L/X}$ has splitting type corresponding to $\Vec{a}$. Then, Section 2 of \cite{hannah} explains that $\Sigma_{\Vec{a}}$ can be endowed with a scheme structure and that the closure in $\Sigma$ $$\overline{\Sigma}_{\Vec{a}} = \bigcup_{\Vec{a}'\le \Vec{a}}\Sigma_{a'},$$
    where $\Vec{a}'\le \Vec{a}$ means that $a'_1+\cdots+a'_k\le a_1+\cdots +a_k$ for all $k.$

    Consider $\Vec{c}$ with
    $$c_{1}=-1, c_{2}=\cdots=c_{d-2}=0, c_{d-1}=\cdots=c_{n-2}=1.$$
    We have a canonical projection map $\pi\colon \overline{\Sigma}_{\Vec{c}} \to U$, and Theorem 1.1 of \cite{hannah} implies that $\overline{\Sigma}_{\Vec{c}}$ is irreducible and has codimension $n-d\ge 0$, which means that it has dimension $N+2n-d-3 - (n-d) = N+n-3$. Moreover, since $\overline{\Sigma}_{\Vec{c}}$ is proper over $\C$ and $U$ is separated over $\C$, it follows that $\pi$ is proper. By Theorem 1.3 of \cite{hannah}, any fiber of $\pi$ above a general hypersurface of degree $d$ is non-empty. Then, since $\pi$ is proper, it is moreover surjective, which proves the first claim about the existence of unbalanced splittings. 
    
    Moreover, by surjectivity, the dimension of any fiber is at least $\dim \pi^{-1}(U) - \dim U = N+n-3 - N = n-3$. Finally, each $L\subset X$ determines a map $\varphi\colon\P^{1}\to X$ up to an action of $\Aut(\P^{1})=\PGL_{2}$, which has dimension $3$. Thus, the dimension of the locus of all such $\varphi$ is $n$ as desired.
    
\end{proof}
   

\begin{proof}[Proof of Proposition \ref{reductthm} for $d\geq 3$]
    For non-smoothness, it suffices to show that there exist two points in the same irreducible component of $\Mor_{e}(C,X)$ with tangent spaces of different dimensions. 

    Let $\phi\colon C\to\P^{1}$ be a map of degree $e$, which exists by Lemma \ref{constructmap}. Let $\varphi_{1},\varphi_{2}\in\Mor_1(\P^{1}, X)$ be as in Lemma \ref{balanced} and Lemma \ref{unbalanced} respectively, and let $f_{1}=\phi\circ\varphi_{1}$ and $f_{2}=\phi\circ\varphi_{2}$. More specifically, by Lemma \ref{balanced}, we may assume that $\varphi_1$ and $\varphi_2$ are in the same irreducible component $Y$ of $\Mor_{1}(\P^{1},X)$. 
    
    Define $x(\Vec{a})=\#\{a_{i}=1\}$ and $y(\Vec{a})=\#\{a_{i}=0\}$. Since all $a_{i}\leq 1$ and $\sum_{i=1}^{n-2}a_{i}=n-d-1$, we have $x(\Vec{a})-(n-2-x(\Vec{a})-y(\Vec{a}))\geq n-d-1$ and so $2x(\Vec{a})+y(\Vec{a})\geq 2n-d-3$.
    
    Note that $\im(\varphi_{1})\in F_{\Vec{b}}(X)$ for $\Vec{b}$ with $x(\Vec{b})=n-d-1$ and $y(\Vec{b})=d-1$, and $\im(\varphi_{2})\in F_{\Vec{c}}(X)$ for $\Vec{c}$ with $x(\Vec{c})\geq n-d$.   
    
    For any $\varphi_{j}$ with $\im(\varphi_{j})\in F_{\Vec{a}}(X)$ and $f_{j}=\phi\circ\varphi_{j}$, we have 
    \begin{align*}
        h^{0}(C,f_{j}^{*}\mathcal{T}_{X}) &=
        h^{0}(C,\phi^{*}\mathcal{O}(2))+
        h^{0}(C,\phi^{*}\mathcal{O}(1))x(\Vec{a})+
        h^{0}(C,\phi^{*}\mathcal{O})y(\Vec{a})+
        \sum_{a_{i}<0}h^{0}(C,\phi^{*}\mathcal{O}(a_{i})) \\&=
        (2e+1-g)+(e+1-g)x(\Vec{a})+y(\Vec{a}).
    \end{align*}
    Thus, since $e+1-g\geq 3$ and $x(\Vec{c})-x(\Vec{b})\geq1$,
    \begin{align*}
        h^{0}(C,f_{2}^{*}\mathcal{T}_{X})-h^{0}(C,f_{1}^{*}\mathcal{T}_{X}) &=
        (e+1-g)(x(\Vec{c})-x(\Vec{b}))+(y(\Vec{c})-y(\Vec{b})) \\&>
        2(x(\Vec{c})-x(\Vec{b}))+(y(\Vec{c})-y(\Vec{b})) \\& =
        2x(\Vec{c})+y(\Vec{c})-(2x(\Vec{b})+y(\Vec{b}))\\&\geq
        2n-d-3-(2(n-d-1)+d-1) \\&= 0.
    \end{align*}
        
    By construction, $f_{1},f_{2}\colon C\to X$ live in a family over the irreducible variety $Y\hookrightarrow \Mor_{1}(\P^{1},X)$. Indeed, consider the composition
    $$\mathcal{F}\colon C\times Y\xrightarrow{\phi\times\textnormal{id}}\P^{1}\times Y\to X,$$
    where the latter morphism is the universal map. Then $\mathcal{F}(-,\varphi_{j})=f_{j}$ as desired. 

    Since the dimension of the tangent space increases at singularities, the singular locus of $\Mor_{e}(C,X)$ contains all points corresponding to maps $f_2=\phi\circ\varphi_{2}$ where $\im(\varphi_{2})\in F_{\Vec{c}}(X)$ and $\phi\colon C\to \P^{1}$ has degree $e$. By Lemma \ref{unbalanced}, the dimension of the locus of choices for $\varphi_{2}$ is at least $n$. And by Lemma \ref{constructmap} the dimension of the locus of choices of $\phi$ is the Brill-Noether number $-g+2e-2$. Thus, the locus of all choices of $\phi$ has dimension at least $n+2e-g-2$.
\end{proof}

\begin{proof}[Proof of Theorem \ref{mainthm} for $d\geq 3$]
    Follows immediately from Proposition \ref{reductthm} for $d\geq 3$ and Lemma \ref{reductlem}. 
\end{proof}

\bibliography{bibliography}{}
\bibliographystyle{plainurl}

\end{document}